\documentclass[12pt]{article}
\usepackage[english]{babel}
\usepackage{amssymb}
\usepackage{graphics}

\newcommand{\qed}{$\;\;\;\Box$}
\newenvironment{proof}{\par\smallbreak{\sl\bf Proof.~}}
{\unskip\nobreak\hfill \qed \par\medbreak}
\newenvironment{proofthm}{\par\smallbreak{\sl\bf Proof of  Theorem~}}
{\unskip\nobreak\hfill \qed \par\medbreak}
\newcounter{claim}
\renewcommand{\theclaim}{\arabic{claim}}
{\par\medskip\par}

{\qed\par\smallbreak}
\newcommand{\N}{{\mathbb N}}
\newcommand{\R}{{\mathbb R}}

\newcommand{\beq}{\begin{equation}}
\newcommand{\ee}{\end{equation}}

\renewcommand{\d}{\partial}

\newtheorem{thm}{Theorem}[section]

\newtheorem{defn}[thm]{Definition}

\newtheorem{ex}[thm]{Example}

\newcommand{\io}{\iota}

\newcounter{e}
\newcommand{\reff}[1]{(\ref{#1})}

\title{Bounded Solutions to Boundary Value \\ Hyperbolic Problems} 

\newcounter{thesame}
\author{
R. Klyuchnyk
\thanks{Institute for Applied Problems of Mechanics and Mathematics,
Ukrainian National Academy of Sciences, {\small
  E-mail:
{\tt roman.klyuchnyk@gmail.com}}}
\ \ \ I. Kmit \thanks{Humboldt University of Berlin,
{\small   E-mail:
{\tt kmit@mathematik.hu-berlin.de}}. On leave from the
Institute for Applied Problems of Mechanics and Mathematics,
Ukrainian National Academy of Sciences.
}}

\date{}

\begin{document}

\maketitle

\begin{abstract}
\noindent

We investigate linear boundary value problems for first-order one-dimensional hyperbolic systems in a strip. We establish conditions for existence and uniqueness of bounded continuous solutions. For that we suppose that the non-diagonal part of the zero-order coefficients vanish at infinity. Moreover, we establish a dissipativity condition in terms of the boundary data and the diagonal part of the zero-order coefficients.

\end{abstract}



\section{Introduction}\label{sec:intr} 

\renewcommand{\theequation}{{\thesection}.\arabic{equation}}
\setcounter{equation}{0}

\subsection{Problem setting and our result}\label{sec:setting}

We investigate the general linear first-order hyperbolic system in a single space variable

 \begin{equation}
 \partial_{t}u_j
 +a_j(x,t)\partial_{x}u_j
 +\sum_{k=1}^{n}b_{jk}(x,t)u_k
 =f_j(x,t),
 \;\;\; (x,t)\in(0,1)\times\R,\;\;\; j\le n,
 \label{f1}
 \end{equation}
 subjected to the boundary conditions
 \beq\label{f2}
 \begin{array}{ll}
   u_{j}(0,t)=
 (Ru)_j(t),
 \;\;\; 1\le j\le m,\; t\in \R,  \\ [2mm]
  u_{j}(1,t)=
 (Ru)_j(t),
 \;\;\; m< j\le n,\; t\in \R,
 \end{array}
 \ee
where $0\le m\le n$ are positive integers and $R=(R_1,...,R_n)$ is a linear bounded operator from $BC_n$ to $BC_n(\R)$. Here and below by $BC_n$ we will denote the vector space of all bounded and continuous maps $u:[0,1]\times\R\to\R^{n}$, with the norm
 $$
 \|u\|_{\infty}=\max_{j\le n}\max_{x\in [0,1]}\sup_{t\in\R}|u_j|.
 $$
 Similarly, by $BC_n^1$ we will denote a Banach space of all $u\in BC_n$ such that $\d_xu,\d_tu\in BC_n$, with the norm 
 $$
\|u\|_{1}=\|u\|_{\infty}+\|\partial_x u\|_{\infty}+\|\partial_t u\|_{\infty}.
 $$
Also, we use the notation $BC_n(\R)$ for the space of all bounded and continuous maps $v:\R\to\R^n$ and the notation $BC_n^1(\R)$ for the space of all $v\in BC_n(\R)$ with $v^\prime\in BC_n(\R)$. For simplification, we will skip subscript $n$ if $n=1$.

We make the following assumptions on the coefficients of \reff{f1}:
 \beq\label{f4}
 a_j, b_{jk}\in BC^1 \mbox{ for all } j\le n \mbox{ and } k\le n,
 \ee
\beq\label{f5}
 \inf_{j,x,t}a_j>0  \mbox{ for all } j\le m \;\;\;\mbox{and}\;\;\; \sup_{j,x,t}a_j<0  \mbox{ for all } m<j\le n.
\ee
Suppose also that
\beq\label{ft7}
\begin{array}{ll}
 \mbox{the restriction of the operator } R \mbox{ to } BC_n^1
 \\  
 \mbox{is a linear bounded operator from } BC_n^1 \mbox{ to } BC_n^{1}(\R)
 \end{array}
 \ee
 and 
\beq\label{fz8}
\mbox{for all } 1\le j\neq k\le n \mbox{ there exists } \tilde{b}_{jk}\in BC^1 \mbox{ such that }
b_{jk}=\tilde{b}_{jk}(a_k-a_j).
\ee

Let us introduce the characteristics of the hyperbolic system \reff{f1}. Given $j\le n$, $x\in[0,1]$, and $t\in\R$, the $j$-th characteristic is defined as the solution $\xi\in[0,1]\mapsto\omega_j(\xi,x,t)\in\R$ of the initial value problem

\beq\label{f7}
\d_{\xi}\omega_{j}(\xi,x,t)=\frac{1}{a_j(\xi,\omega_{j}(\xi,x,t))}, \;\;\; \omega_{j}(x,x,t)=t.
\ee
To shorten notation, we will simply write $\omega_j(\xi)=\omega_j(\xi,x,t)$. Set 

\beq\label{f8}
c_j(\xi,x,t)=
\exp{{\int_{x}^{\xi}\left (\frac{b_{jj}}{a_j}\right )(\eta,\omega_j(\eta)) d\eta}}, \;\;\;
d_j(\xi,x,t)=
\frac{c_j(\xi,x,t)}{a_j(\xi,\omega_j(\xi))}.
\ee
Integration along the characteristic curves brings the system \reff{f1}--\reff{f2} to the integral form
\begin{eqnarray}
\lefteqn{
 u_j(x,t)=c_j(0,x,t)(Ru)_j(w_j(0))}\nonumber\\ [2mm] &&
 -\int_{0}^{x}d_j(\xi,x,t)\sum_{k\neq j} b_{jk}(\xi,\omega_j(\xi))u_k(\xi,\omega_j(\xi)) d\xi
 +\int_{0}^{x}d_j(\xi,x,t)f_j(\xi,\omega_j(\xi)) d\xi,\nonumber
 \\
&& \hskip10cm 1\le j\le m, \label{f10}
 \\
\lefteqn{
  u_j(x,t)=c_j(1,x,t)(Ru)_j(w_j(1))}\nonumber
 \\ [2mm] &&
 -\int_{1}^{x}d_j(\xi,x,t)\sum_{k\neq j} b_{jk}(\xi,\omega_j(\xi))u_k(\xi,\omega_j(\xi)) d\xi
 +\int_{1}^{x}d_j(\xi,x,t)f_j(\xi,\omega_j(\xi)) d\xi\nonumber
 \\
 && \hskip10cm m<j\le n.
 \label{f11}
\end{eqnarray}
By straightforward calculation, one can easily show that a $C^1$- map $u:[0,1]\times\R\to\R^{n}$ is a solution to the PDE problem \reff{f1}--\reff{f2} if and only if it satisfies the system \reff{f10}--\reff{f11}.
This motivates the following definition.

\begin{defn}
A function $u\in BC_n$ is called a bounded continuous solution to \reff{f1}--\reff{f2} if it satisfies \reff{f10} and \reff{f11}.
\end{defn}
Introduce an operator $C:BC_n\to BC_n$ by 
\beq\label{f12}
 (Cv)_j(x,t)=\left\{
 \begin{array}{rl}
c_j(0,x,t)(Rv)_j(\omega_j(0)) &\mbox{for}\ 1\le j\le m,\\
c_j(1,x,t)(Rv)_j(\omega_j(1)) &\mbox{for}\ m<j\le n.
\end{array}
\right.
 \ee

\begin{thm}\label{thm:th11}
Suppose that the conditions \reff{f4}--\reff{fz8} are fulfilled.
Moreover, assume that there exists $\ell\in\N$ such that
\beq\label{f14}
\|C^{\ell}\|_{\mathcal{L}(BC_n)}<1,
\ee
and
\beq\label{ebjk}
\begin{array}{ll}
 \mbox{for all } \varepsilon>0 \mbox{ there exists a compact interval } I\subset\R \mbox{ such that } 
 \\ 
 |b_{jk}(x,t)|<\varepsilon \mbox{ for all } 1\le j\neq k\le n, \; x\in[0,1] \mbox{ and } t\in\R\setminus I.
 \end{array}
 \ee
Then the problem \reff{f1}--\reff{f2} has a unique bounded continuous solution $u$.
\end{thm}

\begin{ex}\rm
 The following example shows that if the conditions \reff{fz8} and \reff{ebjk} are not satisfied, then the statement of Theorem \ref{thm:th11} is not true, in general. Specifically, we consider the problem
\begin{equation}
 \begin{array}{ll}
 \displaystyle\partial_{t}u_1
 +\frac{2}{\pi}\partial_{x}u_1
 -u_2=0 &  \\ [2mm]
\displaystyle\partial_{t}u_2
 +\frac{2}{\pi}\partial_{x}u_2
 +u_1=0,  &
 \end{array}
 \label{f1ex}
 \end{equation} 
 \beq\label{f2ex}
 \begin{array}{ll}
  u_{1}(0,t)=0,  \;\;\;\; 
  u_{2}(1,t)=0. \\  &
  \end{array}
  \ee
Obviously, \reff{f1ex}--\reff{f2ex} is a particular case of \reff{f1}--\reff{f2} and satisfies all assumptions of Theorem \ref{thm:th11} with the exception of \reff{fz8} and \reff{ebjk}. It is straightforward to check that
$$
u_1=\sin{\frac{\pi}{2}x}\sin{l(t-\frac{\pi}{2}x)}, \;\;\; u_2=\cos{\frac{\pi}{2}x}\sin{l(t-\frac{\pi}{2}x)}, \; l\in\mathbb{N}
$$
are infinitely many linearly independent bounded continuous 
($2\pi$-periodic in $t$) solutions to the problem \reff{f1ex}--\reff{f2ex}.
This means that the kernel of the operator of \reff{f1ex}--\reff{f2ex} is infinite dimensional. 
Thus, the uniqueness conclusion of Theorem \ref{thm:th11} is not true.
\end{ex}

In Section \ref{sec:motiv} we give a brief motivation of our investigations. 
In Section \ref{sec:Fredh} we prove the Fredholm alternative for  \reff{f1}--\reff{f2} 
(Theorem \ref{thm:th12}), while in Section \ref{sec:uni} -- a uniqueness result 
for \reff{f1}--\reff{f2} (Theorem \ref{thm:thuni}). Theorem \ref{thm:th11} will 
then straightforwardly follow from Theorems \ref{thm:th12} and \ref{thm:thuni}.

\subsection{Motivation and state of the art}\label{sec:motiv}

Systems of the type \reff{f1}--\reff{f2}
are used to model problems
of laser dynamics \cite{LiRadRe,RadWu,Sieber},  chemical kinetics \cite{zel}, chemotaxis \cite{Seg}
and population dynamics  \cite{HRL}. Another area of applications of such models is boundary control problems \cite{coron,lakra}. In many mathematical models the system \reff{f1} is controlled by the so-called reflection boundary 
conditions what is a particular case of \reff{f2}.

In \cite{Coppel} Coppel proved
the Dichotomy theorem for the linear ODE, namely $x^\prime=A(t)x$. 
It claims that the inhomogeneous equation $x^\prime=A(t)x+f(t)$ has a unique bounded continuous solution on $\R$ for every bounded and continuous function $f$ if and only if the homogeneous equation $x^\prime=A(t)x$ has an exponential dichotomy on $\R$. In \cite{Latt} the authors provide a criterion of the existence of exponential dichotomy on $\R$ for a strongly continuous exponentially bounded evolution family on a Banach space in terms of existence and uniqueness of a bounded continuous mild solutions. In this respect, our result is an important step towards the existence of the exponential dichotomy for boundary value hyperbolic problems.

The correct posedness 
of a particular case 
of \reff{f1}--\reff{f2} was investigated in \cite{Kirilich}. Specifically, 
the authors  studied the system \reff{f1} with the boundary conditions
\beq\label{f1100}
\begin{array}{ll}
u_j(0,t)=\mu_j(t), \; j\le m,  \\ [2mm]
u_j(1,t)=\mu_j(t), \; m<j\le n.
\end{array}
\ee
and investigate existence and uniqueness of continuous but {\it not necessarily bounded solutions}
solutions.
The main assumption imposed in \cite{Kirilich} 
is a smallness of all $b_{jk}$ in a neighborhood of $-\infty$. It comes from the 
Banach fixed point argument used in the proof of the main result. 
In comparison, in the present paper, we 
allow for $b_{jj}$ to be elements of $BC$ only, and for $b_{jk}$ with $j\neq k$ we impose the assumption \reff{ebjk}. 
Moreover, after the changing of variables $u_j\to v_j=u_j-\mu_j(t)$ in \reff{f1} and \reff{f1100} we get $C=0$.
This means that the dissipativity conditions \reff{f14} is satisfied here automatically (with $\ell=1$).

In \cite{KmKl,KR2,KR3}  time-periodic solutions to the system \reff{f1}
with reflection boundary conditions are investigated.
It is suggested a rather general approach  to proving the Fredholm alternative
in spaces of time-periodic functions 
(in the autonomous case  \cite{KR2})
and in the space of continuous and time-periodic functions 
(in the non-autonomous case  \cite{KR3}). In the present paper, we extend this approach from the spaces of periodic functions to
the spaces of bounded functions and prove the Fredholm alternative
for quite general boundary conditions. Note that, when the problem \reff{f1}--\reff{f2}
is considered in the space of continuous and periodic in time functions, then
according to the standard terminology \reff{f14} 
means non-resonant behavior of  \reff{f1}--\reff{f2}.

\section{Fredholm alternative}\label{sec:Fredh}

\renewcommand{\theequation}{{\thesection}.\arabic{equation}}
\setcounter{equation}{0}

We use the notation 
$$
 x_j=\left\{
 \begin{array}{rl}
 0 &\mbox{if}\ 1\le j\le m,\\
 1 &\mbox{if}\ m< j\le n.
\end{array}
\right.
$$
Define linear bounded operators $D,F: BC_n\to BC_n$ by
 \beq \label{f34}
 (Du)_j(x,t)=
 -\int_{x_j}^{x}d_j(\xi,x,t)\sum_{k\neq j} b_{jk}(\xi,\omega_j(\xi))u_k(\xi,\omega_j(\xi)) d\xi, \;\;\; j\le n
 \ee
 and
 \beq \label{f35}
 (Ff)_j(x,t)=
 \int_{x_j}^{x}d_j(\xi,x,t)f_j(\xi,\omega_j(\xi)) d\xi, \;\;\; j\le n.
 \ee
 On the account of \reff{f12}, \reff{f34}, and \reff{f35}, the system \reff{f10}--\reff{f11} can be written as the operator equation
 \beq \label{f31}
 u=Cu+Du+Ff.
 \ee

\begin{thm}\label{thm:th12}
Suppose that all conditions of Theorem \ref{thm:th11} are fulfilled. Let $\mathcal{K}$ denote the vector space of all bounded continuous solutions to 
\reff{f1}--\reff{f2} with $f\equiv0$. Then

$(i)$ $\dim \mathcal{K}<\infty$ and the vector space of all $f\in BC_{n}$ such that there exists a bounded continuous solution to \reff{f1}--\reff{f2} 
is a closed subspace of codimension $\dim \mathcal{K}$ in $BC_{n}$.

$(ii)$ If $\dim \mathcal{K}=0$, then for any $f\in BC_{n}$ there exists a unique bounded continuous solution $u$ to \reff{f1}--\reff{f2}.
\end{thm}

One of the technical tools we employ for the proof is a generalized Arzela-Ascoli compactness criteria for unbounded domains, see \cite{Lev}. To formulate it, we need a corresponding notion of equicontinuity.
\begin{defn}
A family $\Phi\subset BC_{n}$ is called equicontinuous on $[0,1]\times\R$ if
\begin{itemize}
\item $\Phi$ is equicontinuous on any compact set in $[0,1]\times\R$, and 
\item for any $\varepsilon>0$ there exists $T>0$ such that \begin{equation}\label{qwq}
|u(x^{'},t^{'})-u(x^{''},t^{''})|<\varepsilon
\end{equation}
for all $x^{'}, x^{''}\in [0,1]$, all $t^{'}, t^{''}\in\R\setminus[-T,T]$, and all $u\in\Phi$.
\end{itemize}
\end{defn}

\begin{thm}\label{thm:th21} (\rm a generalized Arzela-Ascoli theorem)
A family $\Phi\subset BC_{n}$ is precompact in $BC_{n}$ if and only if $\Phi$ is bounded in $BC_{n}$ and equicontinuous on $[0,1]\times\R$.
\end{thm}

\begin{proofthm}{\bf\ref{thm:th12}.}
First observe that the operator $I-C:BC_{n}\to BC_{n}$ is bijective, what straightforwardly follows from the Banach fixed-point theorem and the condition \reff{f14}. Then the operator $I-C-D$ is Fredholm of index zero if and only if
\begin{equation}\label{fricd}
I-(I-C)^{-1}D:BC_n\to BC_n \;\textrm{is Fredholm of index zero.}
\end{equation}
Nikolsky's criterion \cite[Theorem XIII.5.2]{KA} says that an operator $I+K$ on a Banach space is Fredholm of index zero whenever $K^2$ is compact.
Hence, we are done with \reff{fricd} if we show that the operator$[(I-C)^{-1}D]^2=(I-C)^{-1}D(I-C)^{-1}D$ is compact. As the composition of a compact and a bounded operator is a compact operator, it is enough to show that  
$$
D(I-C)^{-1}D:BC_n\to BC_n \;\textrm{is compact.}
$$
Since $D(I-C)^{-1}D=D^2+DC(I-C)^{-1}D$ and $(I-C)^{-1}D$ is bounded, it is sufficient to prove that
\beq \label{f38}
D^2, DC:BC_n\to BC_n\;\textrm{are compact.}
\ee
To show \reff{f38}, we use Theorem \ref{thm:th21}. 
Given $T>0$, set $Q(T)=\{(x,t)\in[0,1]\times\R\,:\,-T\le t\le T\}$. Fix an arbitrary bounded set $X\subset BC_n$. For \reff{f38} it is sufficient to prove the following two statements:
\begin{equation}\label{plusone}
D^2X \;\textrm{and}\; DCX\;\textrm{are equicontinuous on }\; Q(T)\;\textrm{ for an arbitrary fixed }\;T>0
\end{equation}
and
\begin{equation}\label{plustwo}
\begin{array}{cc}
\textrm{given}\; \varepsilon>0, \;\textrm{there exists}\; T>0 \;\textrm{such that \reff{qwq}}\; \textrm{is fulfilled for all}\; &  \\ 
x^{'}, x^{''}\in[0,1],\;\;\; t^{'}, t^{''}\in\R\setminus[-T,T], \;\;\; u\in D^2X \;\textrm{and}\; u\in DCX. & 
\end{array}
\end{equation}
We start with the proving \reff{plusone}. Denote by $C_n(Q(T))$ (respectively $C_n^1(Q(T))$) the Banach space of continuous (respectively continuously differentiable) vector functions on $Q(T)$.
As $C^{1}_{n}(Q(T))$ is compactly embedded into $C_{n}(Q(T))$ (due to the Arzela-Ascoli theorem), it is sufficient to show that
\beq \label{f39}
\|D^2u|_{Q(T)}\|_{C^{1}_{n}(Q(T))}+\|DCu|_{Q(T)}\|_{C^{1}_{n}(Q(T))}=O(\|u\|_{\infty}) \;\textrm{for all }\; u\in X.
\ee
It should be noted that  for all  sufficiently large~$T$ the functions 
$D^2u$ and $DCu$ restricted to $Q(T)$ depend  only on $u$ restricted to $Q(2T)$.

We will use the following formulas
\beq \label{f22}
 \partial_x\omega_{j}(\xi)=
 -\frac{1}{a_j(x,t)}\exp{{\int_{\xi}^{x}\left (\frac{\partial_2a_j}{a_j^2}\right )(\eta,\omega_j(\eta)) d\eta}},
 \ee
 \beq \label{f23}
 \partial_t\omega_{j}(\xi)=
 \exp{{\int_{\xi}^{x}\left (\frac{\partial_2a_j}{a_j^2}\right )(\eta,\omega_j(\eta)) d\eta}},
 \ee
 being true for all $j\le n$, all $\xi,x\in[0,1]$, and all $t\in\R$. Here and below by $\d_i$ we denote the partial derivative with respect to the $i$-th argument. Then for all sufficiently large $T>0$ the partial derivatives $\partial_xD^2u$, $\partial_{t}D^2u$, $\partial_{x}DCu$, and $\partial_{t}DCu$  on $Q(T)$ 
exist and are continuous for all $u\in C^{1}(Q(2T))$. Since $C^{1}(Q(2T))$ is dense in $C(Q(2T))$, 
then the desired property \reff{f39} will follow from the  bound
\beq \label{f310}
\left\|D^2u|_{Q(T)}\right\|_{C^1_{n}(Q(T))}
+\|DCu|_{Q(T)}\|_{C^1_{n}(Q(T))}
= O(\|u\|_{C_{n}(Q(2T))}) \;\textrm{for all}\; u\in C^1_{n}(Q(2T)).
\ee
This bound is proved similarly to \cite[Lemma 4.2]{KR3}:

We start with the estimate
$$
\left\|D^2u|_{Q(T)}\right\|_{C^1_{n}(Q(T))}=O(\|u\|_{C_{n}(Q(2T))}) \;\textrm{for all}\; u\in C^1_{n}(Q(2T)).
$$
Given $j\le n$ and $u\in C^1_n(Q(2T))$, let us consider the following representation for $(D^2u)_j(x,t)$ obtained after the application of the Fubini's theorem:
\beq \label{f311}
(D^2u)_j(x,t)
=\sum_{k\neq j}\sum_{l\neq k}\int_{x_j}^x\int_{\eta}^{x} d_{jkl}(\xi,\eta,x,t)b_{jk}(\xi,\omega_j(\xi))u_l(\eta,\omega_k(\eta,\xi,\omega_j(\xi))) d\xi d\eta,
\ee
where 
\beq \label{f311d}
d_{jkl}(\xi,\eta,x,t)=d_j(\xi,x,t)d_{k}(\eta,\xi,\omega_{j}(\xi))b_{kl}(\eta,\omega_{k}(\eta,\xi,\omega_j(\xi))).
\ee
It is easy to see, that from \reff{f311} follows that 
$$
\left\|D^2u|_{Q(T)}\right\|_{C_n(Q(T))}= O(\|u\|_{C_{n}(Q(2T))}).
$$
Since 
$$
(\d_t+a_j(x,t)\d_x)\varphi(\omega_j(\xi,x,t))=0
$$
for all $j\le n, \varphi\in C^1(\R), x,\xi\in[0,1]$, and $t\in\R$, one can easily check that
$$
\|[(\d_t+a_j(x,t)\d_x)(D^2u)_j|_{Q(T)}]\|_{C_{n}(Q(T))} = O\left(\|u\|_{C_n(Q(2T))}\right)
\mbox{ for all } j\le n \mbox{ and } u\in C^1_{n}(Q(2T)).
$$
Hence the estimate 
$\left\|\d_xD^2u|_{Q(T)}\right\|_{C_n(Q(T))}= O(\|u\|_{C_{n}(Q(2T))})$ will follow from the following one:
\beq \label{f31jhr1}
\|\d_tD^2u|_{Q(T)}\|_{C_{n}(Q(T))}= O(\|u\|_{C_{n}(Q(2T))}).
\ee
We are therefore reduced to prove \reff{f31jhr1}. To this end, we start with the following consequence of \reff{f311}:
\begin{eqnarray*}
\lefteqn{
\d_t[(D^2u)_j(x,t)]}
\nonumber\\ &&
=\displaystyle\sum_{k\neq j}\sum_{l\neq k}\int_{x_j}^x\int_{\eta}^{x} \frac{d}{dt}\Bigl[ d_{jkl}(\xi,\eta,x,t)b_{jk}(\xi,\omega_j(\xi))\Bigr] u_l(\eta,\omega_k(\eta,\xi,\omega_j(\xi))) d\xi d\eta
\nonumber\\ &&
+\displaystyle\sum_{k\neq j}\sum_{l\neq k}\int_{x_j}^x\int_{\eta}^{x} d_{jkl}(\xi,\eta,x,t) b_{jk}(\xi,\omega_j(\xi))
\nonumber\\ &&
\times
\d_t\omega_k(\eta,\xi,\omega_j(\xi))\d_t\omega_j(\xi)\d_2u_l(\eta,\omega_k(\eta,\xi,\omega_j(\xi))) d\xi d\eta. \label{f312}
\end{eqnarray*}
Let us transform the second summand. Using \reff{f7}, \reff{f22}, and \reff{f23}, we get
\begin{eqnarray}
\lefteqn{
\frac{d}{d\xi} u_l(\eta,\omega_k(\eta,\xi,\omega_j(\xi)))} \nonumber \\ &&
=\Bigl[\d_x\omega_k(\eta,\xi,\omega_j(\xi))+\d_t\omega_k(\eta,\xi,\omega_j(\xi))\d_{\xi}\omega_j(\xi)\Bigr] \d_2u_l(\eta,\omega_k(\eta,\xi,\omega_j(\xi))) \label{eqwn}
\\ &&
=\left ( \frac{1}{a_j(\xi,\omega_j(\xi))}-\frac{1}{a_k(\xi,\omega_j(\xi))}\right ) \d_t\omega_k(\eta,\xi,\omega_j(\xi))\d_2u_l(\eta,\omega_k(\eta,\xi,\omega_j(\xi))). \nonumber
\end{eqnarray}
Therefore, 
\begin{eqnarray}
\lefteqn{ 
 b_{jk}(\xi,\omega_j(\xi))\d_t\omega_k(\eta,\xi,\omega_j(\xi))\d_2u_l(\eta,\omega_k(\eta,\xi,\omega_j(\xi)))}
\nonumber \\ &&
 =\displaystyle a_j(\xi,\omega_j(\xi))a_k(\xi,\omega_j(\xi))\tilde{b}_{jk}(\xi,\omega_j(\xi))\frac{d}{d\xi} u_l(\eta,\omega_k(\eta,\xi,\omega_j(\xi))), \label{f313}
\end{eqnarray}
where the functions $\tilde{b}_{jk}\in BC$ are fixed to satisfy \reff{fz8}. Note that $\tilde{b}_{jk}$ are not uniquely defined by \reff{fz8} for $(x,t)$ with $a_{j}(x,t)=a_{k}(x,t)$. Nevertheless, as it follows from \reff{eqwn}, the right-hand side (and, hence, the left-hand side of \reff{f313}) do not depend on the choice of $\tilde{b}_{jk}$, since $\frac{d}{d\xi}u_{l}(\eta,\omega_{k}(\eta,\xi,\omega_{j}(\xi)))=0$ if $a_{j}(x,t)=a_{k}(x,t)$. 

Write
$$
\tilde{d}_{jkl}(\xi,\eta,x,t)
=d_{jkl}(\xi,\eta,x,t)\d_t\omega_j(\xi)a_k(\xi,\omega_j(\xi))a_j(\xi,\omega_j(\xi))\tilde{b}_{jk}(\xi,\omega_j(\xi)),
$$
where $d_{jkl}$ are introduced by \reff{f311d} and \reff{f8}. Using  \reff{f7} and \reff{f22}, 
we see that the function $\tilde{d}_{jkl}(\xi,\eta,x,t)$ is $C^1$-smooth
 in $\xi$ due to the regularity assumptions \reff{f4} and \reff{fz8}. Similarly,
using \reff{f23}, we see that the functions $d_{jkl}(\xi,\eta,x,t)$ and $b_{jk}(\xi,\omega_j(\xi))$ are $C^1$-smooth in $t$.

By \reff{f313} we have
\begin{eqnarray}\label{parDt}
\lefteqn{
\d_t[(D^2u)_j(x,t)]} \nonumber
\\ &&
= \displaystyle \sum_{k\neq j}\sum_{l\neq k}\int_{x_j}^x\int_{\eta}^{x} \frac{d}{dt} [d_{jkl}(\xi,\eta,x,t)b_{jk}(\xi,\omega_j(\xi))] u_l(\eta,\omega_k(\eta,\xi,\omega_j(\xi))) d\xi d\eta \nonumber
\\ &&
+\displaystyle \sum_{k\neq j}\sum_{l\neq k}\int_{x_j}^x\int_{\eta}^{x}\tilde{d}_{jkl}(\xi,\eta,x,t)\frac{d}{d\xi} u_l(\eta,\omega_k(\eta,\xi,\omega_j(\xi))) d\xi d\eta
\nonumber
\\ &&
=\displaystyle \sum_{k\neq j}\sum_{l\neq k}\int_{x_j}^x\int_{\eta}^{x} \frac{d}{dt} [d_{jkl}(\xi,\eta,x,t)b_{jk}(\xi,\omega_j(\xi))] u_l(\eta,\omega_k(\eta,\xi,\omega_j(\xi))) d\xi d\eta
\\ &&
-\displaystyle \sum_{k\neq j}\sum_{l\neq k}\int_{x_j}^x\int_{\eta}^{x}\d_{\xi}\tilde{d}_{jkl}(\xi,\eta,x,t)u_l(\eta,\omega_k(\eta,\xi,\omega_j(\xi))) d\xi d\eta
\nonumber
\\ &&
+\displaystyle \sum_{k\neq j}\sum_{l\neq k}\int_{x_j}^x\left [\tilde{d}_{jkl}(\xi,\eta,x,t) u_l(\eta,\omega_k(\eta,\xi,\omega_j(\xi)))\right ]_{\xi=\eta}^{\xi=x} d\eta. 
\nonumber
\end{eqnarray}
The desired estimate \reff{f31jhr1} now easily follows from the assumptions \reff{f4}, \reff{f5}, \reff{fz8} and the equations \reff{f311} and \reff{parDt}.

To finish with \reff{f39}, it remains to show that 
\beq \label{f314}
\|DCu|_{Q(T)}\|_{C_{n}(Q(T))}+\|\d_tDCu|_{Q(T)}\|_{C_{n}(Q(T))}= O(\|u\|_{C_{n}(Q(2T))}) \;\textrm{for all} \; u\in C^1_n(Q(2T)), 
\ee
as the estimate for $\d_xDCu$ follows similarly to the case of $\d_xD^2u$. In order to prove \reff{f314}, we consider an arbitrary integral contributing into $DCu$, namely
\beq \label{f315}
\int_{x}^{x_j} e_{jk}(\xi,x,t)b_{jk}(\xi,\omega_j(\xi))(Ru)_k(\omega_k(x_k,\xi,\omega_j(\xi))) d\xi,
\ee
where 
$$
e_{jk}(\xi,x,t)=d_j(\xi,x,t)c_k(x_k,\xi,\omega_j(\xi))
$$
and $j\le n$ and $k\le n$ are arbitrary fixed. From \reff{f315} follows that
$$
\|DCu|_{Q(T)}\|_{C_{n}(Q(T))}=O(\|u\|_{C_{n}(Q(2T))}).
$$
Differentiating \reff{f315} in $t$, we get
\begin{eqnarray}
\lefteqn{
\displaystyle \int_{x}^{x_j} \frac{d}{dt}\Bigl[e_{jk}(\xi,x,t)b_{jk}(\xi,\omega_j(\xi))\Bigr](Ru)_{k}(\omega_k(x_k,\xi,\omega_j(\xi))) d\xi}
\label{dtDC}
\\ &&
+\displaystyle \int_{x}^{x_j} e_{jk}(\xi,x,t)b_{jk}(\xi,\omega_j(\xi))
\d_t\omega_k(x_k,\xi,\omega_j(\xi))\d_t\omega_j(\xi)\d_{2}(Ru)_k(\omega_k(x_k,\xi,\omega_j(\xi))) d\xi.\nonumber
\end{eqnarray}
Our task is to estimate the second integral; for the first one the desired estimate is obvious. Similarly to the above, we use \reff{f7}, \reff{f22}, and \reff{f23} to obtain 
\begin{eqnarray*}
\lefteqn{
\frac{d}{d\xi}(Ru)_k(\omega_k(x_k,\xi,\omega_j(\xi)))}
\\ &&
=\Bigl[\d_x\omega_k(x_k,\xi,\omega_j(\xi))+\d_t\omega_k(x_k,\xi,\omega_j(\xi))\d_{\xi}\omega_j(\xi)\Bigr]\d_{2}(Ru)_{k}(\omega_k(x_k,\xi,\omega_j(\xi)))
\\ &&
=\left ( \frac{1}{a_j(\xi,\omega_j(\xi))}-\frac{1}{a_k(\xi,\omega_j(\xi))}\right ) \d_t\omega_k(x_k,\xi,\omega_j(\xi))\d_{2}(Ru)_{k}(\omega_k(x_k,\xi,\omega_j(\xi))).
\end{eqnarray*}
Taking into account \reff{fz8}, the last expression reads
\begin{eqnarray}
\lefteqn{ 
 b_{jk}(\xi,\omega_j(\xi))\d_t\omega_k(x_k,\xi,\omega_j(\xi))\d_{2}(Ru)_{k}(\omega_k(x_k,\xi,\omega_j(\xi)))}
\nonumber \\ &&
 =\displaystyle a_j(\xi,\omega_j(\xi))a_k(\xi,\omega_j(\xi))\tilde{b}_{jk}(\xi,\omega_j(\xi))\frac{d}{d\xi}(Ru)_{k}(\omega_k(x_k,\xi,\omega_j(\xi))).\label{f313009}
\end{eqnarray}
Set
$$
\tilde{e}_{jk}(\xi,x,t)
=e_{jk}(\xi,x,t)\d_t\omega_j(\xi)a_k(\xi,\omega_j(\xi))a_j(\xi,\omega_j(\xi))\tilde{b}_{jk}(\xi,\omega_j(\xi)).
$$
Using \reff{f22} and \reff{f313009}, let us transform the second summand in \reff{dtDC} as
\begin{eqnarray}
%\lefteqn{
&\displaystyle \int_{x}^{x_j} e_{jk}(\xi,x,t)b_{jk}(\xi,\omega_j(\xi))
\d_t\omega_k(x_k,\xi,\omega_j(\xi))\d_t\omega_j(\xi)\d_{2}(Ru)_{k}(\omega_k(x_k,\xi,\omega_j(\xi))) 
d\xi&
%}
\nonumber\\ &\displaystyle
=\int_{x}^{x_j}\tilde{e}_{jk}(\xi,x,t)\frac{d}{d\xi} (Ru)_{k}(\omega_k(x_k,\xi,\omega_j(\xi))) d\xi&
\nonumber\\ &\displaystyle
=\Bigl[\tilde{e}_{jk}(\xi,x,t) (Ru)_k(\omega_k(x_k,\xi,\omega_j(\xi)))\Bigr]_{\xi=x}^{\xi=x_j}&
\nonumber\\ &\displaystyle
-\int_{x}^{x_j}\d_{\xi}\tilde{e}_{jk}(\xi,x,t)(Ru)_k(\omega_k(x_k,\xi,\omega_j(\xi))) d\xi.&
\label{fc22}
\end{eqnarray}
The bound \reff{f314} now easily follows from \reff{dtDC} and \reff{fc22}. This finishes the proof of the bound \reff{f310} and, hence the statement \reff{plusone}.

It remains to prove \reff{plustwo}. Fix an arbitrary $\varepsilon>0$. We have to prove the estimates
\begin{equation}\label{d2e}
|(D^2u)(x^{'},t^{'})-(D^{2}u)(x^{''},t^{''})|<\varepsilon
\end{equation}
and 
\begin{equation}\label{dce}
|(DCu)(x^{'},t^{'})-(DCu)(x^{''},t^{''})|<\varepsilon
\end{equation}
for all $u\in X$ and all $x^{'}, x^{''}\in [0,1]$,  $t^{'}, t^{''}\in\R\setminus[-T,T]$ and some $T>0$.

Let us prove \reff{d2e}. By \reff{f311}, given $j\le n$ and $u\in X$, we have
$$
\begin{array}{cc}
|(D^2u)_j(x^{'},t^{'})-(D^{2}u)_j(x^{''},t^{''})|\le|(D^2u)_j(x^{'},t^{'})|+|(D^{2}u)_j(x^{''},t^{''})| &  \\
=2\displaystyle\max\limits_{j\le n}\max\limits_{x\in[0,1]}\max\limits_{t\in\R\setminus[-T,T]}\left|\sum_{k\neq j}\sum_{l\neq k}\int_{x_j}^{x}\int_{\eta}^{x} d_{jkl}(\xi,\eta,x,t)b_{jk}(\xi,\omega_j(\xi))u_l(\eta,\omega_k(\eta,\xi,\omega_j(\xi))) d\xi d\eta\right| &  \\
\le M\|u\|_{\infty}\displaystyle\max\limits_{k\neq j, l\neq k}\max\limits_{x,\xi,\eta\in[0,1]}\max\limits_{t\in\R\setminus[-T,T]}|b_{jk}(\xi,\omega_j(\xi))b_{kl}(\eta,\omega_k(\eta,\xi,\omega_j(\xi)))|, & 
\end{array}
$$
the constant $M$ being dependent on $n$, $a_j$ and $b_{jj}$ but not on $u\in X$ and $b_{jk}$ with $j\neq k$. Since $\|u\|_{\infty}$ is bounded on $X$, the desired estimate \reff{d2e} now straightforwardly follows from the assumption \reff{ebjk} and the fact that $\omega_j(\xi,x,t)\to\infty$ as $t\to\pm\infty$.

The estimate \reff{dce} can be obtained by the same argument, what finishes the proof of \reff{plustwo}. The theorem is proved.

\end{proofthm}

\section{Uniqueness of a bounded continuous solution}\label{sec:uni}

\begin{thm}\label{thm:thuni}
Suppose that the conditions \reff{f4}, \reff{f5}, and \reff{f14} are fulfilled. Then there is $\varepsilon>0$ and $T>0$ such a bounded continuous solution to \reff{f1}--\reff{f2} (if any) is unique whenever
\beq\label{ebjk2}
 |b_{jk}(x,t)|<\varepsilon \mbox{ for all } 1\le j\neq k\le n, \; x\in[0,1], \mbox{ and } t\in(-\infty,T].
 \ee
\end{thm}
\begin{proof}
Given $T\in\R$, let $\Pi^{T}=[0,1]\times(-\infty,T]$ and $\Pi_{T}=[0,1]\times[T,\infty)$. Given $T>0$, consider the problem \reff{f1}--\reff{f2} in $\Pi^{-T}$. 
The system of integral equations can again be written in the operator form $u=\tilde{C}u+\tilde{D}u+\tilde{F}f$ with operators $\tilde{C},\tilde{D},\tilde{F}:BC_{n}(\Pi^{-T})\to BC_{n}(\Pi^{-T})$ given by the rules \reff{f12}, \reff{f34} and \reff{f35}, respectively. As the operator $I-\tilde{C}: BC_{n}(\Pi^{-T})\to BC_{n}(\Pi^{-T})$ is bijective (by the condition \reff{f14}), the operator equation reads
\begin{equation}\label{fre}
u=(I-\tilde{C})^{-1}\tilde{D}u+(I-\tilde{C})^{-1}\tilde{F}f.
\end{equation}
Because of assumption \reff{ebjk2}, the value of $T>0$ can be chosen so large that the norm of the operator $\tilde{D}$ is sufficiently small. Consequently, for such $T$ we have
$$
\|(I-\tilde{C})^{-1}\tilde{D}\|_{\mathcal{L}(BC_n(\Pi^{-T}))}<1.
$$
By the Banach fixed-point theorem, there exists a unique function $u\in BC_n(\Pi^{-T})$ satisfying \reff{f10}--\reff{f11} in $\Pi^{-T}$.

Now consider the problem \reff{f1}--\reff{f2} in $\Pi_{-T}$ with the initial condition
\begin{equation}\label{trysta}
u_{j}|_{t=-T}=u_j(x,-T), \; j\le n.
\end{equation}
Existence and uniqueness of a continuous solution $u\in C_n(\bar{\Pi}_{-T})$ to the 
initial-boundary value problem \reff{f1}, \reff{f2}, \reff{trysta} follows from \cite{Km}. 
Summarizing, the problem \reff{f1}, \reff{f2} in the strip $[0,1]\times\R$  has a unique 
continuous solution  bounded at $-\infty$.
This immediately entails that a bounded continuous solution to the problem \reff{f1}--\reff{f2} 
(if any) is unique. The proof is therewith complete.
\end{proof}

 To finish with Theorem \ref{thm:th11}, it remains to note that, by   Theorem \ref{thm:thuni},
$\dim \mathcal{K}=0$. Then Theorem \ref{thm:th11} immediately follows from  Theorem \ref{thm:th12} $(\io\io)$.

\end{document}